\documentclass[a4paper,12pt]{article}
\makeatletter
    
    \@addtoreset{equation}{section}
  \makeatother
\usepackage{amsmath,amssymb,amsthm}
\newtheorem{prop}{Proposition}[section]
\newtheorem{thm}[prop]{Theorem}
\newtheorem{lem}[prop]{Lemma}
\newtheorem{cor}[prop]{Corollary}
\begin{document}
\title{On the global monodromy of a Lefschetz fibration
arising from the Fermat surface of degree 4}
\author{Yusuke Kuno}
\date{}
\maketitle

\begin{abstract}
A complete description of the global monodromy of a Lefschetz fibration
arising from the Fermat surface of degree $4$ is given.
As a by-product we get a positive relation among right hand Dehn twists
in the mapping class group of a closed orientable surface of genus $3$.
\end{abstract}

\section{Introduction}
The motivation of this work is an interest in the topological monodromy of surface bundles obtained by the following
way. Let $X\subset \mathbb{P}_N$ be a complex surface embedded in the complex projective space of dimension $N$.
We denote by $\mathbb{P}^N$ the dual projective space of $\mathbb{P}_N$, i.e., the space of all hyperplanes of
$\mathbb{P}_N$. The dual variety $X^{\vee}$ of $X$ is, by definition, the set of all hyperplanes of $\mathbb{P}_N$
tangent to $X$ at some point. Then we have a complex analytic family of compact Riemann surfaces over
$\mathbb{P}^N \setminus X^{\vee}$; the fiber over $H\in \mathbb{P}^N \setminus X^{\vee}$ is the hyperplane section
$H\cap X$.

If we regard such a family as an oriented surface bundle, its bundle structure is totally encoded
(at least when the genus of $H\cap X$ is $\ge 2$) in the associated topological monodromy $\rho$ from the fundamental
group $\pi_1(\mathbb{P}^N \setminus X^{\vee})$, which is non-trivial when $X^{\vee}$ is a hypersurface,
to the mapping class group $\Gamma_g$ of a closed orientable surface of genus $g$, where $g$ is the genus of
$H\cap X$. If a finite presentation of $\pi_1(\mathbb{P}^N \setminus X^{\vee})$ and a description of $\rho$
in terms of this presentation are obtained, we might say that the topological monodromy $\rho$ is understood.
However such a nice situation may not be expected in general. One reason for this is the difficulty of
the computations of $\pi_1(\mathbb{P}^N \setminus X^{\vee})$, see \cite{DL}.

Instead we consider to cut $\mathbb{P}^N \setminus X^{\vee}$ by a generic line. Let $L$
be a line ($1$-dimensional projective subspace) of $\mathbb{P}^N$ and consider the
restriction of the family over $\mathbb{P}^N\setminus X^{\vee}$ to $L\setminus (L\cap X^{\vee})$.
We focus on the associated topological monodromy $\rho^{\prime}$ from
$\pi_1(L\setminus (L\cap X^{\vee}))$ to $\Gamma_g$. If $L$ meets $X^{\vee}$
transversely, $L\cap X^{\vee}$ consists of finitely many points and the inclusion
$L \setminus (L\cap X^{\vee})\hookrightarrow \mathbb{P}^N \setminus X^{\vee}$
induces the surjection on the fundamental group level (the Zariski theorem of
Lefschetz type, see \cite{H}). Thus for instance to know the group ${\rm Im}(\rho)$,
called \textit{the universal monodromy group} in \cite{DL}, it suffices to consider
$\rho^{\prime}$ instead of $\rho$. Moreover, theory of Lefschetz pencils can be
applied to the study of $\rho^{\prime}$, as follows. There is a natural family of
algebraic curves over $L$; the fiber over $H\in L$ is the (possibly singular) hyperplane
section $H\cap X$. As in \cite{Katz} or \cite{L}, this family turns out to be
a Lefschetz fibration in the sense of \cite{GS}, Definition 8.1.4. In particular all
the singular fibers, which are over $L\cap X^{\vee}$, have one nodal singularity and
the local monodromy around each point of $L\cap X^{\vee}$ is the right hand Dehn twist
along a simple closed curve, called \textit{the vanishing cycle}. Thus the determination
of the positions of all the vanishing cycles on a fixed reference fiber will lead to a complete
description of \textit{the global monodromy} $\rho^{\prime}$. Also, as a by-product we will get a
positive relation among right hand Dehn twists in $\Gamma_g$, since
$\pi_1(L\setminus L\cap X^{\vee})$ admits a presentation by a standard generating system
subject to one relation, see the paragraph before Theorem \ref{thm:1-1}.

In this paper we investigate a particular example. Hereafter $X$ is the Fermat surface of degree $4$,
namely the smooth hypersurface in $\mathbb{P}_3$ defined by the equation
$$x_0^4+x_1^4+x_2^4+x_3^4=0,$$
where $[x_0:x_1:x_2:x_3]$ is a homogeneous coordinate system of $\mathbb{P}_3$.
In this case $X^{\vee}$ is an irreducible hypersurface of $\mathbb{P}^3$,
whose defining equation will be given in section 2. Let
$$\mathcal{F}:=\left\{ (x,H)\in \mathbb{P}_3 \times (\mathbb{P}^3\setminus X^{\vee});\ 
x\in H\cap X \right\}$$
and let
\begin{equation}
\rho \colon \pi_1(\mathbb{P}^3 \setminus X^{\vee},b_0)\rightarrow \Gamma_3
\label{eq:1-1}
\end{equation}
be the associated topological monodromy of the second projection
$\pi \colon \mathcal{F}\rightarrow \mathbb{P}^3 \setminus X^{\vee}$,
where $b_0$ is a base point. Note that for each $H\in \mathbb{P}^3 \setminus X^{\vee}$,
the hyperplane section $\pi^{-1}(H)=H\cap X \subset H\cong \mathbb{P}_2$
is a non-singular plane curve of degree 4. 

To state the result we prepare a terminology. Let $\{ v_i\}_i$ be a set of $n$ points
in $\mathbb{P}_1$ and choose a base point $b_0$ of $\mathbb{P}_1 \setminus \{ v_i\}_i$.
We say a set of $n$ based loops $\{ \lambda_i \}_i$ is a
\textit{standard generating system for $\pi_1(\mathbb{P}_1\setminus \{ v_i\}_i,b_0)$}
if each $\lambda_i$ is free homotopic to a loop nearby $v_i$ going once around
$v_i$ by counter-clockwise manner, and their product $\lambda_1\lambda_2\cdots \lambda_n$
is trivial as an element of $\pi_1(\mathbb{P}_1\setminus \{ v_i\}_i,b_0)$.

\begin{thm}
\label{thm:1-1}
Let $X$ be the Fermat surface of degree $4$ and $L$ a line of $\mathbb{P}^3$ meeting
$X^{\vee}$ transversely. Choose a base point $b_0$ of $L\setminus (L\cap X^{\vee})$.
Then there is a standard generating system $\lambda_1,\ldots,\lambda_{36}$ for
$\pi_1(L\setminus (L\cap X^{\vee}),b_0)$ such that the monodromy $\rho^{\prime}(\lambda_i)$
is given by the right hand Dehn twist along the simple closed curve $C_i$ on a
genus $3$ surface as shown in Figure $1.1$. Here, $\rho^{\prime}$ is the composition of
$\pi_1(L\setminus (L\cap X^{\vee}),b_0)\rightarrow \pi_1(\mathbb{P}^3\setminus X^{\vee},b_0)$
induced by the inclusion and $(\ref{eq:1-1})$.
\end{thm}

Since $\lambda_1\lambda_2\cdots \lambda_{36}=1$ and $\rho^{\prime}$
is an anti-homomorphism (see the conventions below), we immediately have the following
\begin{cor}
Let us denote by $t_i$ the right hand Dehn twist along $C_i$. Then the relation $t_{36}t_{35}\cdots t_2t_1=1$
holds in the mapping class group $\Gamma_3$.
\end{cor}

Also we can show the following
\begin{cor}
The topological monodromy $(\ref{eq:1-1})$ is surjective.
In other words, the universal monodromy group ${\rm Im}(\rho)$ coincides with $\Gamma_3$.
\end{cor}

\begin{proof}
The set of the right hand Dehn twists along the seven simple closed curves $C_9$,
$t_9(C_8)$, $C_{28}$, $t_{28}^{-1}(C_1)$, $C_{10}$, $t_{10}^{-1}(C_{11})$,
and $t_3^{-1}t_{28}^{-1}(C_6)$ constitutes a Dehn-Lickorish-Humphries
generating system of $\Gamma_3$ (see \cite{I}, Corollary 4.2.F).
Thus $\rho^{\prime}$ is surjective, so is $\rho$.
\end{proof}

\begin{center}
Figure 1.1
\input{figureprime1-1-1.tex}
\input{figureprime1-1-2.tex}
\input{figureprime1-1-3.tex}
\input{figureprime1-1-4.tex}
\end{center}

\newpage

\begin{center}
Figure 1.1(continued)
\input{figureprime1-1-5.tex}
\input{figureprime1-1-6.tex}
\end{center}

The study of the global monodromy of a holomorphic fibration of Riemann surfaces over a Riemann surface
via numerical analysis is initiated by Ahara \cite{A} and Matsumoto \cite{M}. They introduced
a holomorphic fibration $f\colon V_n\rightarrow \mathbb{P}_1$, where $V_n$ is the Fermat surface
of degree $n$. Their fibration is not a Lefschetz fibration and has more degenerated singular fibers.
Their method was to express the
general fibers as branched coverings of $\mathbb{P}_1$ and analyze the motions of the critical
points of these branched coverings. The analysis is based on Newton approximation, see \cite{A}, section 3.
Based on the result of \cite{A}, the global monodromy was described in terms of Dehn twists
for the case $n=5$ in \cite{M} (in this case the genus of the general fibers is 3).
Recently, Ahara and Awata \cite{AA} determined how general fibers of $f$ degenerate
to the singular fibers for all $n$, without numerical analysis.

The rest of the paper is devoted to the proof of Theorem \ref{thm:1-1}.
Note that the total space of our Lefschetz fibration $\bar{\pi}\colon \bar{\mathcal{F}}\rightarrow L$
(see section 5) is the blow up of $X$ at 4 points. We adopt the same method as \cite{A}, \cite{M}.
In section 2 we give the defining equation of $X^{\vee}$. In section 3 we cut
$\mathbb{P}^3 \setminus X^{\vee}$ by a line $L=L(c_1,c_2)$ whose defining equation has
two parameters $c_1$ and $c_2$. For a suitable choice of $c_1$ and $c_2$, $L$ will meet
$X^{\vee}$ transversely. We will introduce a homogeneous coordinate system $[u:v]$
to $L$ and will denote by $\mathbb{C}_L$ the set $L\setminus \{ [0:1]\}$. Then
we proceed to express general fibers as 4-branched coverings of $\mathbb{P}_1$.
In section 4 we introduce a projection $p_v$ from $X_v=v \cap X$
to $\mathbb{P}_1$ for $v=[1:v]\in \mathbb{C}_L\subset \mathbb{P}^3$
and prove its "tameness" over $[1:0]\in \mathbb{P}_1$ (see Lemma \ref{lem:4-1}).
Section 5 is a preparation for sections 6 and 7. We choose explicit values for $c_1$
and $c_2$. Most of the results in sections 6 and 7 depend on numerical analysis using a computer.
In section 6 we describe the projection $p_0\colon X_0 \rightarrow \mathbb{P}_1$ and
in section 7 we analyze motions of the critical values of $p_v$ caused by movements
of $v$ along suitable chosen paths in $L$ and give a complete description
of the topological monodromy $\rho^{\prime} \colon \pi_1(L\setminus (L\cap X^{\vee}))\rightarrow
\Gamma_3$. Theorem \ref{thm:1-1} will easily follow from Proposition \ref{prop:7-1}.

\noindent \textbf{Conventions about topological monodromy}

It is sometimes confusing that there are different kinds of conventions about
product of paths or product of maps, so let us fix the conventions in this paper:
1) for any two mapping classes $f_1$ and $f_2$, the multiplication $f_1\circ f_2$ means
that $f_2$ is applied first,
2) for any two homotopy classes of based loops $\ell_1$ and $\ell_2$, their product
$\ell_1\cdot \ell_2$ means that $\ell_1$ is traversed first.

Let $\Sigma$ be a closed oriented surface and $\pi\colon E\rightarrow B$ an oriented $\Sigma$-bundle.
Choose a base point $b_0\in B$ and fix an identification
$\phi\colon \Sigma \stackrel{\cong}{\rightarrow} \pi^{-1}(b_0)$. For each based loop
$\ell\colon [0,1]\rightarrow B$, consider the pull back $\ell^*(E)\rightarrow [0,1]$.
Since $[0,1]$ is contractible there exists a trivialization
$\Phi\colon \Sigma \times [0,1]\rightarrow \ell^*(E)$ such that $\Phi(x,0)=\phi(x)$.
By assigning the isotopy class of $\phi^{-1}\circ \Phi(x,1)$ to the homotopy class of $\ell$,
we obtain a map $\rho$, called \textit{the topological monodromy of $\pi\colon E\rightarrow B$},
from $\pi_1(B,b_0)$ to the mapping class group of $\Sigma$.
Under the conventions above, $\rho$ is an \textit{anti-homomorphism}, i.e.,
for $\ell_1,\ell_2\in \pi_1(B,b_0)$ we have
$$\rho(\ell_1\ell_2)=\rho(\ell_2)\rho(\ell_1).$$

\section{The defining equation of $X^{\vee}$}
Our first task is to describe the defining equation of $X^{\vee}$. The result
might be known, but we give it here since our numerical analysis
by a computer performed in sections 6 and 7 will heavily use it.
To begin with, we compute the degree of $X^{\vee}$. By using the formula of
Katz \cite{Katz} (5.5.1), it is computed as
$${\rm deg}(X^{\vee})=\int_X\frac{(1+h)^2}{c([X])}=\int_X\frac{(1+h)^2}{1+4h}=36.$$
Here, $h\in H^2(\mathbb{P}_3;\mathbb{Z})$ denotes the hyperplane class and
$c([X])\in H^*(\mathbb{P}_3;\mathbb{Z})$ denotes the total Chern class of the divisor $X$.
Let $[\alpha_0:\alpha_1:\alpha_2:\alpha_3]$ be the homogeneous coordinate system
of $\mathbb{P}^3$ dual to $[x_0:x_1:x_2:x_3]$. Namely,
$[\alpha_0:\alpha_1:\alpha_2:\alpha_3]\in \mathbb{P}^3$ is the hyperplane of $\mathbb{P}_3$
defined by $$\alpha_0x_0+\alpha_1x_1+\alpha_2x_2+\alpha_3x_3=0.$$

\begin{prop}
\label{prop:2-1}
Let $\omega=\exp (2\pi \sqrt{-1}/3)$ and
let $\beta_i$ be a formal indeterminate such that $\beta_i^3=\alpha_i$ for
$i=0,1,2$, and $3$.
Then the defining equation of $X^{\vee}$ is given by
\begin{equation}
\label{eq:2-1}
\prod_{0\le i_1,i_2,i_3\le 2}(\beta_0^4+\omega^{i_1}\beta_1^4+\omega^{i_2}\beta_2^4
+\omega^{i_3}\beta_3^4)=0.
\end{equation}
\end{prop}

Remark that the left hand side of (\ref{eq:2-1}) is invariant under the
transformations $\{ \phi_i\}_{i=0}^3$ where $\phi_i$ is defined by
$\phi_i(\beta_j)=\omega\beta_j$ for $j\neq i$ and $\phi_i(\beta_i)=\beta_i$.
Thus it is in fact a homogeneous polynomial in $\alpha_i$'s and the degree is $36$.

\begin{proof}[Proof of Proposition \ref{prop:2-1}]
Since we know the degree of $X^{\vee}$ is also equal to $36$, it suffices to
show that $\alpha \in X^{\vee}$ if and only if $\alpha$ satisfies the equation (\ref{eq:2-1}).
Let $\alpha=[\alpha_0:\alpha_1:\alpha_2:\alpha_3]\in \mathbb{P}^3$ and assume that
$\alpha_0=1$. Let $P(x_0,x_1,x_2,x_3)=x_0^4+x_1^4+x_2^4+x_3^4$. By definition, $\alpha \in X^{\vee}$
if and only if there exists a point $y=[y_0:y_1:y_2:y_3]\in \mathbb{P}_3$ such that
\begin{equation}
\label{eq:2-2}
\begin{cases}
P(y_0,y_1,y_2,y_3)=0 \\
y_0+\alpha_1y_1+\alpha_2y_2+\alpha_3y_3=0 \\
\left[P_{x_0}(y):P_{x_1}(y):
P_{x_2}(y):P_{x_3}(y)\right]=
[1:\alpha_1:\alpha_2:\alpha_3],
\end{cases}
\end{equation}
where $P_{x_0}$ is the partial derivative of $P$ with respect to $x_0$, etc.
Since $P_{x_i}=4x_i^3$ we see that $y_0\neq 0$, by the third equation of (\ref{eq:2-2}).
Thus we may assume $y_0=1$ and we have
\begin{equation}
\label{eq:2-3}
y_1^3=\alpha_1,\ y_2^3=\alpha_2,\ y_3^3=\alpha_3.
\end{equation}
Under (\ref{eq:2-3}), the first and the second equations of (\ref{eq:2-2})
are equivalent. Therefore $\alpha \in X^{\vee}$ if and only if there exists
$(y_1,y_2,y_3)\in \mathbb{C}^3$ such that
$$\begin{cases}
y_1^3=\alpha_1,\ y_2^3=\alpha_2,\ y_3^3=\alpha_3 \\
1+\alpha_1y_1+\alpha_2y_2+\alpha_3y_3=0.
\end{cases}$$
Let $\beta_i$ be a complex number such that $\beta_i^3=\alpha_i$ for $i=1,2,$ and $3$.
Then, $\alpha \in X^{\vee}$ if and only if there exist $i_1,i_2,i_3\in \{ 0,1,2\}$ such that
$$1+\omega^{i_1}\beta_1^4+\omega^{i_2}\beta_2^4+\omega^{i_3}\beta_3^4=0,$$
namely $\beta=[1:\beta_1:\beta_2:\beta_3]$ satisfies the equation (\ref{eq:2-1}).
This completes the proof.
\end{proof} 

\section{Cutting $X^{\vee}$ by a line of special type}
Let $c_1$ and $c_2$ be complex numbers and $L=L(c_1,c_2)$ the line of $\mathbb{P}^3$ defined by
$$c_1^3\alpha_0-\alpha_1=c_2^3\alpha_0-\alpha_2=0.$$
We introduce a homogeneous coordinate system $[u:v]$ of $L$ by assigning
$[\alpha_0:\alpha_1:\alpha_2:\alpha_3]=[u:c_1^3u:c_2^3u:v]$ to $[u:v]$.

\begin{prop}
The defining equation of $L\cap X^{\vee}\subset L$ is given by
$$\prod_{0\le i_1,i_2 \le 2}\left((1+\omega^{i_1}c_1^4+\omega^{i_2}c_2^4)^3u^4
+v^4 \right)=0.$$
\label{prop:3-1}
\end{prop}

\begin{proof}
Let $\beta_i,0\le i\le 3$ be the formal elements as in
Proposition \ref{prop:2-1} and suppose $c_1\beta_0-\beta_1=c_2\beta_0-\beta_2=0$. Then
$\beta_0^4+\omega^{i_1}\beta_1^4+\omega^{i_2}\beta_2^4+\omega^{i_3}\beta_3^4$ is equal to
$$\beta_0^4+\omega^{i_1}(c_1\beta_0)^4+\omega^{i_2}(c_2\beta_0)^4+\omega^{i_3}\beta_3^4
=(1+\omega^{i_1}c_1^4+\omega^{i_2}c_2^4)\beta_0^4+\omega^{i_3}\beta_3^4,$$
and
\begin{eqnarray*}
\prod_{0\le i_3 \le 2}(1+\omega^{i_1}c_1^4+\omega^{i_2}c_2^4)\beta_0^4+\omega^{i_3}\beta_3^4
&=& ((1+\omega^{i_1}c_1^4+\omega^{i_2}c_2^4)\beta_0^4)^3+(\beta_3^4)^3 \\
&=& (1+\omega^{i_1}c_1^4+\omega^{i_2}c_2^4)^3\alpha_0^4+\alpha_3^4.
\end{eqnarray*}
Note that $\omega$ is a primitive third root of unity.
Combining this computation with Proposition \ref{prop:2-1}, we have the result.
\end{proof}
Suppose $c_1$ and $c_2$ are chosen so that
\begin{enumerate}
\item
for any pair $(i_1,i_2)$, we have $1+\omega^{i_1}c_1^4+\omega^{i_2}c_2^4\neq 0$,
\item
for any two distinct pairs $i=(i_1,i_2)$ and $j=(j_1,j_2)$, the roots of
$f_i(v)=v^4+(1+\omega^{i_1}c_1^4+\omega^{i_2}c_2^4)^3$ and those of
$f_j(v)=v^4+(1+\omega^{j_1}c_1^4+\omega^{j_2}c_2^4)^3$ are all different.
\end{enumerate}
Then by Proposition \ref{prop:3-1}, $L\cap X^{\vee}$ consists of ${\rm deg}(X^{\vee})=36$ points
therefore $L$ meets $X^{\vee}$ transversely. Moreover, $L\cap X^{\vee}$ is contained in $L\setminus \{ [1:0],[0:1]\}$.
For simplicity we write $\mathbb{C}_L$ instead of $L\setminus \{ [0:1]\}$, and we identify $\mathbb{C}_L$
with $\mathbb{C}$ by $v\mapsto [1:v], v\in \mathbb{C}$. Choose $0\in \mathbb{C}_L$ as a base point of
$L\setminus (L\cap X^{\vee})$. By the Zariski theorem of Lefschetz type \cite{H}, (for our purpose,
a weaker statement in \cite{L}, (7.4.1) is sufficient) the natural homomorphism
\begin{equation}
\label{eq:3-1}
\pi_1(L\setminus (L\cap X^{\vee}),0)\rightarrow \pi_1(\mathbb{P}^3\setminus X^{\vee},b_0)
\end{equation}
induced by the inclusion is surjective (we denote by $b_0$ the image of
$0\in \mathbb{C}_L$ by the inclusion).
From now on, we assume that $c_1$ and $c_2$ satisfy the two conditions above
and will focus on the surface bundle
$$\pi^{\prime}\colon \mathcal{F}^{\prime}\rightarrow L\setminus (L\cap X^{\vee})$$
where $\mathcal{F}^{\prime}=\pi^{-1}(L\setminus (L\cap X^{\vee}))$ and
$\pi^{\prime}=\pi |_{\mathcal{F}^{\prime}}$. The associated topological monodromy
$$\rho^{\prime}\colon \pi_1(L\setminus (L\cap X^{\vee}),0)
\rightarrow \Gamma_3$$ is the composition of (\ref{eq:3-1}) and (\ref{eq:1-1}).

\section{A lemma on the hyperplane section by $v\in \mathbb{C}_L$}
Let $v\in \mathbb{C}_L$. We denote by $X_v$ the hyperplane section $v\cap X$, whose defining equation is
$$\begin{cases}
x_0^4+x_1^4+x_2^4+x_3^4=0 \\
x_0+c_1^3x_1+c_2^3x_2+vx_3=0.
\end{cases}$$
Eliminating the indeterminate $x_0$, we obtain
$$(c_1^3x_1+c_2^3x_2+vx_3)^4+x_1^4+x_2^4+x_3^4=0.$$
Let $E_v=E_v(x_1,x_2,x_3)$ be the left hand side of this equation.
Then by regarding $[x_1:x_2:x_3]$ as a homogeneous coordinate system of $\mathbb{P}_2$,
$X_v$ is identified with the plane curve determined by $E_v$. Under this identification,
consider the projection
$$p_v\colon X_v\rightarrow \mathbb{P}_1, [x_1:x_2:x_3]\mapsto [x_1:x_3].$$

\begin{lem}
\label{lem:4-1}
If $|c_1|^4+|c_2|^4<1$, the following holds: for any $v \in \mathbb{C}_L$,
\begin{enumerate}
\item the plane curve $X_v$ has no singularities on the line $x_3=0$, and
\item the projection $p_v$ does not branch over $[1:0]\in \mathbb{P}^1$.
\end{enumerate}
\end{lem}

\begin{proof}
For simplicity, we write $E$ instead of $E_v$.
Suppose $E=E_{x_1}=E_{x_2}=E_{x_3}=0$ has a solution $[x_1:x_2:0]$ for some $v \in \mathbb{C}_L$.
If $v \neq 0$, we have $c_1^3x_1+c_2^3x_2=0$ since $E_{x_3}=0$.
Substituting this into $E_{x_1}=E_{x_2}=0$ we have $x_1=x_2=0$,
a contradiction. Thus it suffices to consider the case when $v=0$. Suppose $x_2=1$. Then we have
\begin{equation}
\label{eq:4-1}
\begin{cases}
E_{x_1}=4c_1^3(c_1^3x_1+c_2^3)^3+4x_1^3=0 \\
E_{x_2}=4c_2^3(c_1^3x_1+c_2^3)^3+4=0.
\end{cases}
\end{equation}
By the second equation of (\ref{eq:4-1}), we have
\begin{equation}
\label{eq:4-2}
(c_1^3x_1+c_2^3)^3=-c_2^{-3}.
\end{equation}
Substituting this into the first equation of (\ref{eq:4-1}), we have
$x_1^3=(c_1/c_2)^3$ therefore we can write $x_1=\omega^j c_1/c_2$ for some $j$, $0\le j\le 2$.
Substituting this into (\ref{eq:4-2}) we have a necessary condition
$(c_1^4\omega^j+c_2^4)^3=-1$. But this is impossible by our assumption
$|c_1|^4+|c_2|^4 < 1$. If we assume $x_1=1$ a similar argument
leads to a contradiction. This establishes the first part.

To show the second part, it suffices to show the following:
for $(x_1,x_3)=(1,0)$, the equation $E=E_{x_2}=0$
does not have any solution in $x_2$. The argument is similar to
the first part. Suppose $x_2\in \mathbb{C}$ satisfies
\begin{equation}
\label{eq:4-5}
\begin{cases}
E=(c_1^3+c_2^3x_2)^4+1+x_2^4=0 \\
E_{x_2}=4c_2^3(c_1^3+c_2^3x_2)^3+4x_2^3=0.
\end{cases}
\end{equation}
By the second equation of (\ref{eq:4-5}), we have
\begin{equation}
\label{eq:4-6}
(c_1^3+c_2^3x_2)^3=-\frac{x_2^3}{c_2^3}.
\end{equation}
Substituting this into the first equation of (\ref{eq:4-5}),
we see that $x_2=\omega^j c_2/c_1$ for some $j$, $0\le j\le 2$.
Substituting this into (\ref{eq:4-6}) we have $(c_1^4+c_2^4\omega^j)^3=-1$, a contradiction.
\end{proof}

\section{A special choice of $c_1$ and $c_2$}
Henceforth, let $c_1=7/8$ and $c_2=3/4$. For this choice, the conditions for $c_1$ and $c_2$
given in section 3 and the assumption of Lemma \ref{lem:4-1} are satisfied.

To study $\rho^{\prime}$ (see section 3) we also consider
$\bar{\mathcal{F}}:=\{ (x,H)\in \mathbb{P}_3\times L;\ x\in H\cap X\}$ and the second projection
$\bar{\pi}\colon \bar{\mathcal{F}} \rightarrow L$. By the transversality of $L$ and $X^{\vee}$, it follows that
$\bar{\mathcal{F}}$ is non-singular and $\bar{\pi}\colon \bar{\mathcal{F}} \rightarrow L$ is a Lefschetz fibration 
(see section 1). The set of critical values of $\bar{\pi}$ is $L\cap X^{\vee}=\{ v_1,\ldots,v_{36} \}$.
For each $v_i$, there is a unique critical point $\tilde{v}_i$ in $\bar{\pi}^{-1}(v_i)$ and for a suitable choice
of local holomorphic coordinates, the projection $\bar{\pi}$ looks like $(z_1,z_2)\mapsto z_1^2+z_2^2$
near $\tilde{v}_i$.
In this local model, the singular fiber $\bar{\pi}^{-1}(v_i)$ looks like $\Sigma_0=\{ z_1^2+z_2^2=0\}$,
which is obtained from the smooth fibers $\Sigma_{\varepsilon}=\{ z_1^2+z_2^2=\varepsilon \}$, $\varepsilon >0$
by collapsing the simple closed curves $C_{\varepsilon}=\{ (x_1,x_2)\in \mathbb{R}^2; x_1^2+x_2^2=\varepsilon \}$.
The curve $C_{\varepsilon}$ is called \textit{the vanishing cycle}.
By the Picard-Lefschetz formula (\cite{GS}, p.295), the local monodromy around each $v_i$ is
the right hand Dehn twist along the corresponding vanishing cycle.

Recall that the defining equation of $X_v=v \cap X$ is
$$(c_1^3x_1+c_2^3x_2+vx_3)^4+x_1^4+x_2^4+x_3^4=0.$$
By Lemma \ref{lem:4-1}, $p_v$ is unramified over $[1:0]\in \mathbb{P}_1$. Thus we focus on $p_v$
restricted to $\mathbb{P}_1\setminus \{ [1:0]\}$, which is identified with $\mathbb{C}$ by
$x_1\mapsto [x_1:1],x_1 \in \mathbb{C}$. Let
$$F_{x_1}^v(x_2):=(c_1^3x_1+c_2^3x_2+v)^4+x_1^4+x_2^4+1$$
and $G^v(x_1)$ the discriminant of $F_{x_1}^v$ regarded as a polynomial in $x_2$ and $Q(v)$
the discriminant of $G^v(x_1)$ regarded as a polynomial in $x_1$. $G^v(x_1)$ is a polynomial of
degree 12 in $x_1$. By definition $v\in \mathbb{C}_L$ is a root of $Q$ if and only if there is a root of
$G^v$ with multiplicity $\ge 2$. As we will see in section 7, $G^{v_i}$ has this property hence
$Q(v_i)=0$ for $i=1,\ldots,36$. Therefore if $v$ is not a root of $Q$ the curve $X_v$ is
non-singular and all the roots of $G^v$, which correspond to the critical values of $p_v$, are simple.
By the Riemann-Hurwitz formula we see that the total branching order of each critical value of $p_v$ is $1$.
This means that over each critical value there is an exactly one critical point of $p_v$, near which
$p_v$ looks like $z\mapsto z^2$ for a suitable choice of local coordinates.

\section{Description of the reference fiber}
In this section, we describe the reference fiber $X_0={\pi^{\prime}}^{-1}(0)$ as
a 4-fold branched covering $p_0\colon X_0\rightarrow \mathbb{P}_1$.
As in the last section we focus on $p_0$ restricted to $\mathbb{P}_1\setminus \{ [1:0]\} \cong \mathbb{C}$.

The roots of $G^0(x_1)$ are numerically computed and we denote them by
$a_1,\ldots,a_{12}$ as shown in the following schematic figure:

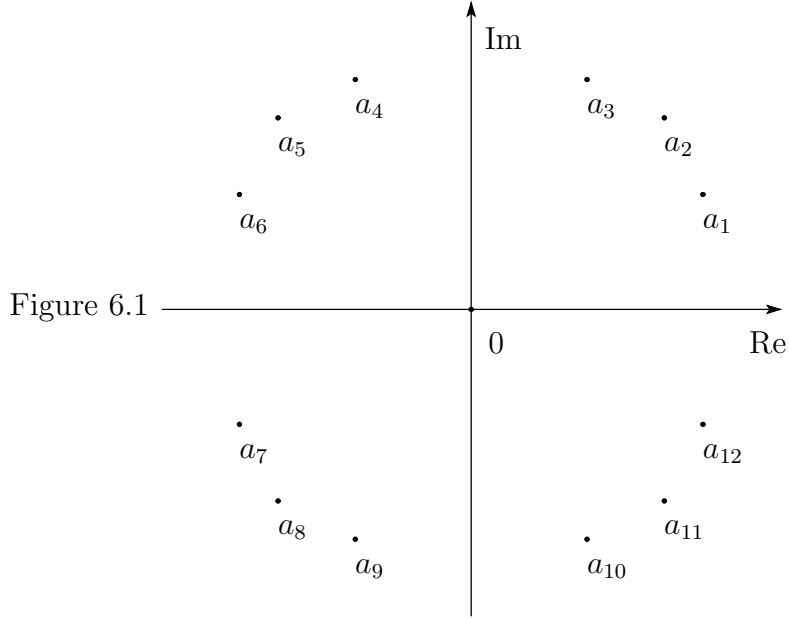
\begin{center}
%WinTpicVersion3.08
\unitlength 0.1in
\begin{picture}( 39.9000, 32.0000)(  0.1000,-34.0000)
% DOT 2 0 3 0
% 2 2400 1800 2400 1800
% 
\special{pn 8}%
\special{sh 1}%
\special{ar 2400 1800 10 10 0  6.28318530717959E+0000}%
\special{sh 1}%
\special{ar 2400 1800 10 10 0  6.28318530717959E+0000}%
% STR 2 0 3 0
% 3 2490 1930 2490 2030 2 0
% $0$
\put(24.9000,-20.3000){\makebox(0,0)[lb]{$0$}}%
% STR 2 0 3 0
% 3 3840 1930 3840 2030 2 0
% Re
\put(38.4000,-20.3000){\makebox(0,0)[lb]{Re}}%
% STR 2 0 3 0
% 3 2470 340 2470 440 2 0
% Im
\put(24.7000,-4.4000){\makebox(0,0)[lb]{Im}}%
% STR 2 0 3 0
% 3 3600 1300 3600 1400 2 0
% $a_1$
\put(36.0000,-14.0000){\makebox(0,0)[lb]{$a_1$}}%
% STR 2 0 3 0
% 3 3400 900 3400 1000 2 0
% $a_2$
\put(34.0000,-10.0000){\makebox(0,0)[lb]{$a_2$}}%
% STR 2 0 3 0
% 3 3000 700 3000 800 2 0
% $a_3$
\put(30.0000,-8.0000){\makebox(0,0)[lb]{$a_3$}}%
% STR 2 0 3 0
% 3 1800 700 1800 800 2 0
% $a_4$
\put(18.0000,-8.0000){\makebox(0,0)[lb]{$a_4$}}%
% STR 2 0 3 0
% 3 1400 900 1400 1000 2 0
% $a_5$
\put(14.0000,-10.0000){\makebox(0,0)[lb]{$a_5$}}%
% STR 2 0 3 0
% 3 1200 1300 1200 1400 2 0
% $a_6$
\put(12.0000,-14.0000){\makebox(0,0)[lb]{$a_6$}}%
% STR 2 0 3 0
% 3 1200 2500 1200 2600 2 0
% $a_7$
\put(12.0000,-26.0000){\makebox(0,0)[lb]{$a_7$}}%
% STR 2 0 3 0
% 3 1400 2900 1400 3000 2 0
% $a_8$
\put(14.0000,-30.0000){\makebox(0,0)[lb]{$a_8$}}%
% STR 2 0 3 0
% 3 1800 3100 1800 3200 2 0
% $a_9$
\put(18.0000,-32.0000){\makebox(0,0)[lb]{$a_9$}}%
% STR 2 0 3 0
% 3 3000 3100 3000 3200 2 0
% $a_{10}$
\put(30.0000,-32.0000){\makebox(0,0)[lb]{$a_{10}$}}%
% STR 2 0 3 0
% 3 3400 2900 3400 3000 2 0
% $a_{11}$
\put(34.0000,-30.0000){\makebox(0,0)[lb]{$a_{11}$}}%
% STR 2 0 3 0
% 3 3600 2500 3600 2600 2 0
% $a_{12}$
\put(36.0000,-26.0000){\makebox(0,0)[lb]{$a_{12}$}}%
% DOT 2 0 3 0
% 2 1800 3000 1800 3000
% 
\special{pn 8}%
\special{sh 1}%
\special{ar 1800 3000 10 10 0  6.28318530717959E+0000}%
\special{sh 1}%
\special{ar 1800 3000 10 10 0  6.28318530717959E+0000}%
% DOT 2 0 3 0
% 2 3000 3000 3000 3000
% 
\special{pn 8}%
\special{sh 1}%
\special{ar 3000 3000 10 10 0  6.28318530717959E+0000}%
\special{sh 1}%
\special{ar 3000 3000 10 10 0  6.28318530717959E+0000}%
% DOT 2 0 3 0
% 2 3600 2400 3600 2400
% 
\special{pn 8}%
\special{sh 1}%
\special{ar 3600 2400 10 10 0  6.28318530717959E+0000}%
\special{sh 1}%
\special{ar 3600 2400 10 10 0  6.28318530717959E+0000}%
% DOT 2 0 3 0
% 2 3600 1200 3600 1200
% 
\special{pn 8}%
\special{sh 1}%
\special{ar 3600 1200 10 10 0  6.28318530717959E+0000}%
\special{sh 1}%
\special{ar 3600 1200 10 10 0  6.28318530717959E+0000}%
% DOT 2 0 3 0
% 2 3000 600 3000 600
% 
\special{pn 8}%
\special{sh 1}%
\special{ar 3000 600 10 10 0  6.28318530717959E+0000}%
\special{sh 1}%
\special{ar 3000 600 10 10 0  6.28318530717959E+0000}%
% DOT 2 0 3 0
% 2 1800 600 1800 600
% 
\special{pn 8}%
\special{sh 1}%
\special{ar 1800 600 10 10 0  6.28318530717959E+0000}%
\special{sh 1}%
\special{ar 1800 600 10 10 0  6.28318530717959E+0000}%
% DOT 2 0 3 0
% 2 1200 1200 1200 1200
% 
\special{pn 8}%
\special{sh 1}%
\special{ar 1200 1200 10 10 0  6.28318530717959E+0000}%
\special{sh 1}%
\special{ar 1200 1200 10 10 0  6.28318530717959E+0000}%
% DOT 2 0 3 0
% 2 1200 2400 1200 2400
% 
\special{pn 8}%
\special{sh 1}%
\special{ar 1200 2400 10 10 0  6.28318530717959E+0000}%
\special{sh 1}%
\special{ar 1200 2400 10 10 0  6.28318530717959E+0000}%
% DOT 2 0 3 0
% 2 1400 2800 1400 2800
% 
\special{pn 8}%
\special{sh 1}%
\special{ar 1400 2800 10 10 0  6.28318530717959E+0000}%
\special{sh 1}%
\special{ar 1400 2800 10 10 0  6.28318530717959E+0000}%
% DOT 2 0 3 0
% 2 3400 2800 3400 2800
% 
\special{pn 8}%
\special{sh 1}%
\special{ar 3400 2800 10 10 0  6.28318530717959E+0000}%
\special{sh 1}%
\special{ar 3400 2800 10 10 0  6.28318530717959E+0000}%
% DOT 2 0 3 0
% 2 3400 800 3400 800
% 
\special{pn 8}%
\special{sh 1}%
\special{ar 3400 800 10 10 0  6.28318530717959E+0000}%
\special{sh 1}%
\special{ar 3400 800 10 10 0  6.28318530717959E+0000}%
% DOT 2 0 3 0
% 2 1400 800 1400 800
% 
\special{pn 8}%
\special{sh 1}%
\special{ar 1400 800 10 10 0  6.28318530717959E+0000}%
\special{sh 1}%
\special{ar 1400 800 10 10 0  6.28318530717959E+0000}%
% VECTOR 2 0 3 0
% 2 800 1800 4000 1800
% 
\special{pn 8}%
\special{pa 800 1800}%
\special{pa 4000 1800}%
\special{fp}%
\special{sh 1}%
\special{pa 4000 1800}%
\special{pa 3934 1780}%
\special{pa 3948 1800}%
\special{pa 3934 1820}%
\special{pa 4000 1800}%
\special{fp}%
% VECTOR 2 0 3 0
% 2 2400 3400 2400 200
% 
\special{pn 8}%
\special{pa 2400 3400}%
\special{pa 2400 200}%
\special{fp}%
\special{sh 1}%
\special{pa 2400 200}%
\special{pa 2380 268}%
\special{pa 2400 254}%
\special{pa 2420 268}%
\special{pa 2400 200}%
\special{fp}%
% STR 2 0 3 0
% 3 10 1760 10 1860 2 0
% Figure 6.1
\put(0.1000,-18.6000){\makebox(0,0)[lb]{Figure 6.1}}%
\end{picture}%
\end{center}

Here, $a_1\approx 0.709187+0.642143\sqrt{-1}$,
$a_2\approx 0.692307+0.692307\sqrt{-1}$,
$a_3\approx 0.642143+0.709187\sqrt{-1}$ and $a_{i+3}=\sqrt{-1}a_i$ for $1\le i \le 9$.

For $x_1\in \mathbb{C}$, the points in the fiber $p_0^{-1}(x_1)$
correspond to the roots of $F_{x_1}^0$ by $[x_1:x_2:1]\mapsto x_2$.
Now we choose $0$ as a base point of $\mathbb{C}\setminus \{a_i \}_i$.
The fiber $p_0^{-1}(0)$ corresponds to the roots of
$$F_0^0(x_2)=(c_2^{12}+1)x_2^4+1,$$ i.e., $\{s_k\}_{k=1}^4$ where
$s_k=(1+c_2^{12})^{-\frac{1}{4}}\exp \left((2k-1)\pi \sqrt{-1}/4\right)$.

We will investigate the monodromy
$$\chi\colon \pi_1(\mathbb{C}\setminus \{a_i\}_i,0)\rightarrow \mathfrak{S}_4$$
of the unramified 4-covering $p_0^{-1}(\mathbb{C}\setminus \{a_i\}_i)\rightarrow
\mathbb{C}\setminus \{a_i\}_i$. Here, $\mathfrak{S}_4$ is the symmetric group on
the four letters $s_1,s_2,s_3$, and $s_4$.

For each $j=1,\ldots,12$, let $m_j$ be the straight line segment from $0$ to $a_j$
and $\ell_j$ be a based loop in $\mathbb{C}\setminus \{a_i\}_i$
going from $0$ to a point nearby $a_j$ along $m_j$, then going
once around $a_j$ by counter-clockwise manner and then coming back
to $0$ along $m_j$, as shown in the following figure.

\begin{center}
%WinTpicVersion3.08
\unitlength 0.1in
\begin{picture}( 32.7000, 13.2000)(  4.3000,-28.0000)
% DOT 2 0 3 0
% 2 1400 1800 1390 1800
% 
\special{pn 8}%
\special{sh 1}%
\special{ar 1400 1800 10 10 0  6.28318530717959E+0000}%
\special{sh 1}%
\special{ar 1390 1800 10 10 0  6.28318530717959E+0000}%
% CIRCLE 2 0 3 0
% 4 3390 1800 3190 1800 3190 1800 3190 1800
% 
\special{pn 8}%
\special{ar 3390 1800 200 200  0.0000000 6.2831853}%
% DOT 2 0 3 0
% 2 3390 1800 3400 1800
% 
\special{pn 8}%
\special{sh 1}%
\special{ar 3390 1800 10 10 0  6.28318530717959E+0000}%
\special{sh 1}%
\special{ar 3400 1800 10 10 0  6.28318530717959E+0000}%
% LINE 2 0 3 0
% 4 1400 1800 1400 1800 3200 1800 3200 1800
% 
\special{pn 8}%
\special{pa 1400 1800}%
\special{pa 1400 1800}%
\special{fp}%
\special{pa 3200 1800}%
\special{pa 3200 1800}%
\special{fp}%
% LINE 2 0 3 0
% 4 3200 1800 3200 1800 1400 1800 3200 1800
% 
\special{pn 8}%
\special{pa 3200 1800}%
\special{pa 3200 1800}%
\special{fp}%
\special{pa 1400 1800}%
\special{pa 3200 1800}%
\special{fp}%
% STR 2 0 3 0
% 3 1400 2100 1400 2200 2 0
% $0$
\put(14.0000,-22.0000){\makebox(0,0)[lb]{$0$}}%
% STR 2 0 3 0
% 3 3200 2100 3200 2200 2 0
% $a_j$
\put(32.0000,-22.0000){\makebox(0,0)[lb]{$a_j$}}%
% VECTOR 2 0 3 0
% 2 3600 2000 3700 1840
% 
\special{pn 8}%
\special{pa 3600 2000}%
\special{pa 3700 1840}%
\special{fp}%
\special{sh 1}%
\special{pa 3700 1840}%
\special{pa 3648 1886}%
\special{pa 3672 1886}%
\special{pa 3682 1908}%
\special{pa 3700 1840}%
\special{fp}%
% STR 2 0 3 0
% 3 2200 1550 2200 1650 2 0
% $\ell_j$
\put(22.0000,-16.5000){\makebox(0,0)[lb]{$\ell_j$}}%
% DOT 2 0 3 0
% 2 1400 2800 1400 2800
% 
\special{pn 8}%
\special{sh 1}%
\special{ar 1400 2800 10 10 0  6.28318530717959E+0000}%
\special{sh 1}%
\special{ar 1400 2800 10 10 0  6.28318530717959E+0000}%
% LINE 2 0 3 0
% 2 1400 2800 3400 2800
% 
\special{pn 8}%
\special{pa 1400 2800}%
\special{pa 3400 2800}%
\special{fp}%
% STR 2 0 3 0
% 3 1400 2610 1400 2710 2 0
% $0$
\put(14.0000,-27.1000){\makebox(0,0)[lb]{$0$}}%
% STR 2 0 3 0
% 3 3200 2610 3200 2710 2 0
% $a_j$
\put(32.0000,-27.1000){\makebox(0,0)[lb]{$a_j$}}%
% STR 2 0 3 0
% 3 2200 2550 2200 2650 2 0
% $m_j$
\put(22.0000,-26.5000){\makebox(0,0)[lb]{$m_j$}}%
% DOT 2 0 3 0
% 2 3400 2800 3400 2800
% 
\special{pn 8}%
\special{sh 1}%
\special{ar 3400 2800 10 10 0  6.28318530717959E+0000}%
\special{sh 1}%
\special{ar 3400 2800 10 10 0  6.28318530717959E+0000}%
% STR 2 0 3 0
% 3 430 2260 430 2360 2 0
% Figure 6.2
\put(4.3000,-23.6000){\makebox(0,0)[lb]{Figure 6.2}}%
\end{picture}%
\end{center}

By numerical analysis using a computer, we see that $\chi(\ell_j)$ is given by the following table:

\begin{center}
\begin{tabular}{c|lcc|l}
$j$ & $\chi(\ell_j)$ & & $j$ & $\chi(\ell_j)$ \\
\hline
1 & (12) & & 7 & (34) \\
2 & (13) & & 8 & (13) \\
3 & (14) & & 9 & (23) \\
4 & (23) & & 10 & (14) \\
5 & (24) & & 11 & (24) \\
6 & (12) & & 12 & (34) \\
\end{tabular}
\end{center}

For example, $\chi(\ell_1)=(12)$ means $\chi(\ell_1)$ is the transposition of $s_1$ and $s_2$, etc.
Let $S_k=[0:s_k:1]$ and $\tilde{a}_j$ the unique critical point of $p_0$ over $a_j$,
and $\tilde{m}_j$ the connected component of $p_0^{-1}(m_j)$ containing
$\tilde{a}_j$ as an interior point.
Then $p_0^{-1}(0)=\{ S_k\}_{k=1}^4$ and we can draw the picture
of $S_k$, $\tilde{a}_j$, and $\tilde{m}_j$ on $X_0$
by using the table above, which determines the topological type
of the branched covering $p_0\colon X_0\rightarrow \mathbb{P}_1$. See the figure below.
\begin{center}
\input{figure6-3.tex}
\end{center}
For example, $\tilde{m}_1$ is the unique path from $S_1$ through $\tilde{a}_1$ to $S_2$,
corresponding to the data $\chi(\ell_1)=(12)$. In section 7 this figure
will be a key to find the vanishing cycles.

\section{Finding the vanishing cycles}
In this section we give a complete description of
$$\rho^{\prime}\colon \pi_1(L\setminus (L\cap X^{\vee}),0)\rightarrow \Gamma_3$$
and finish the proof of Theorem \ref{thm:1-1}.
Our task is to determine the position of all the vanishing cycles in $X_0$.
We will achieve this by investigating the motions of the critical values of $p_v$
along a suitably chosen path from $0$ to each point of
$L\cap X^{\vee}=\{ v_1\ldots,v_{36}\}$.

Now we arrange indices of $v_i$'s and let
$\mu_i\colon [0,1]\rightarrow \mathbb{C}_L$ be a simple path from $0$ to
$v_i$, satisfying $Q(\mu_i(t))\neq 0$ for $t\in [0,1)$, as shown in Figure 7.1.

\begin{center}
\input{figure7-1.tex}
\end{center}

Approximate values of $v_i$'s are:
$v_1\approx 0.600851+0.315483\sqrt{-1}$,
$v_2\approx 0.963952+0.064039\sqrt{-1}$,
$v_3\approx 0.999689+0.470655\sqrt{-1}$,
$v_4\approx 1.059535+0.794167\sqrt{-1}$,
$v_5\approx 1.145495+1.145495\sqrt{-1}$,
$v_i={\rm Im}(v_{10-i})+{\rm Re}(v_{10-i})$
for $6\le i\le9$, and $v_{i+9}=\sqrt{-1}v_i$ for $1\le i\le 27$.
Each $\mu_i$ consists of 4 straight line segments, as shown in the
figure. Here, $\zeta_i$ is a root of $(1+\omega^i c_1^4)^3+\zeta_i^4$ such
that ${\rm Re}(\zeta_i)>0, {\rm Im}(\zeta_i)>0$ for $i=1,2,3$
and $\mu_{i+9}=\sqrt{-1}\mu_i$ for $1\le i\le 27$.

Let $\lambda_i$ be a based loop in $\mathbb{C}_L \setminus (L\cap X^{\vee})$
going from $0$ to a point nearby $v_i$ along $\mu_i$, then going
once around $v_i$ by counter-clockwise manner and then coming back to $0$ along $\mu_i$. 
Then $\{\lambda_1,\ldots,\lambda_{36} \}$ is a standard generating system for
$\pi_1(L\setminus (L\cap X^{\vee}),0)$ in the sense of section 1.

For a while we fix $i,1\le i\le 36$. For each $t\in [0,1)$ the roots of $G^{\mu_i(t)}(x_1)$
are all simple, therefore we can choose complex valued continuous functions
$a_1(t),\ldots,a_{12}(t)$ such that $G^{\mu_i(t)}(a_j(t))=0$ and $a_j(0)=a_j$
for $j=1,\ldots,12$. We have $a_j(t)\neq a_k(t)$ for $t\in [0,1)$ and $(j,k)$ with $j\neq k$.

By continuity, $a_j(t)$ is uniquely extended to a continuous function on the unit interval $[0,1]$.
We would like to study what happens when $t$ approaches $1$.
By numerical analysis using a computer, we can investigate the motions of $a_j(t)$, $1\le j\le 12$.

\noindent \textbf{Observation 1.}
There exist two indices $\delta=\delta(i)$
and $\varepsilon=\varepsilon(i),1\le \delta<\varepsilon\le 12$ such that
$a_{\delta}(1)=a_{\varepsilon}(1)$ and $a_j(1) \neq a_k(1)$ for any pair
$(j,k)$ with $j<k$ other than $(\delta,\varepsilon)$, see the table below.
In particular, the number of roots of $G^{v_i}$ is 11.

\begin{center}
\begin{tabular}{l|ccl|ccl|ccl|c}
$i$ & $(\delta(i),\varepsilon(i))$ & & $i$ &$(\delta(i),\varepsilon(i))$
& & $i$ & $(\delta(i),\varepsilon(i))$ & & $i$ & $(\delta(i),\varepsilon(i))$ \\
\hline
1 & (3,6) & & 10 & (6,9) & & 19 & (9,12) & & 28 & (3,12) \\
2 & (1,4) & & 11 & (4,7) & & 20 & (7,10) & & 29 & (1,10) \\
3 & (2,5) & & 12 & (5,8) & & 21 & (8,11) & & 30 & (2,11) \\
4 & (1,7) & & 13 & (4,10) & & 22 & (1,7) & & 31 & (4,10) \\
5 & (2,8) & & 14 & (5,11) & & 23 & (2,8) & & 32 & (5,11) \\
6 & (3,9) & & 15 & (6,12) & & 24 & (3,9) & & 33 & (6,12) \\
7 & (2,11) & & 16 & (2,5) & & 25 & (5,8) & & 34 & (8,11) \\
8 & (3,12) & & 17 & (3,6) & & 26 & (6,9) & & 35 & (9,12) \\
9 & (1,10) & & 18 & (1,4) & & 27 & (4,7) & & 36 & (7,10) \\
\end{tabular}
\end{center}

\noindent \textbf{Observation 2.}
For any root $a_j(1)$ of $G^{v_i}$,
the number of roots of $F_{a_j(1)}^{v_i}$ is 3.

Let $\{{\gamma_i}^t \colon [0,1]\rightarrow \mathbb{C} \}_{0\le t\le 1}$ be
a continuous family of paths constructed by the following way.
First choose a real number $t_0<1$
sufficiently near $1$, and for $t\in [t_0,1]$, let ${\gamma_i}^t$ be
the straight path joining $a_{\delta}(t)$ and $a_{\varepsilon}(t)$.
Next extending the motions of $a_i(t)$'s for $t\in [0,t_0]$, we have an ambient isotopy
$\tau\colon \mathbb{C}\times [0,t_0]\rightarrow \mathbb{C}$ of $\mathbb{C}$
such that $\tau(x_1,t_0)=x_1$ and $\tau(a_i(t_0),t)=a_i(t)$, $1\le i\le 12$.
Finally we set ${\gamma_i}^t(s)=\tau({\gamma_i}^{t_0}(s),t)$ for $t\in [0,t_0]$.
Note that we may assume that ${\gamma_{i+9}}^t=\sqrt{-1}{\gamma_i}^t$.
This follows from the fact that $x_1\in \mathbb{C}$ is a root of $G^v$ if and
only if $\sqrt{-1}x_1$ is a root of $G^{\sqrt{-1}v}$. Then we have

\noindent \textbf{Observation 3.} ${\gamma_i}^0$, $1\le i\le 9$ look like Figure 7.2.

By construction the family $\{{\gamma_i}^t\}_{0\le t\le 1}$ satisfies the following three conditions:
\begin{enumerate}
\item for each $t\in [0,1]$, ${\gamma_i}^t(0)=a_{\delta}(t)$ and ${\gamma_i}^t(1)=a_{\varepsilon}(t)$,
\item for each $t\neq 1$, ${\gamma_i}^t$ is a simple path not meeting
$\{ a_j(t)\}_{j\neq \delta, \varepsilon}$,
\item ${\gamma_i}^1(s)=a_{\delta}(1)=a_{\varepsilon}(1)$, for $s\in [0,1]$.
\end{enumerate}

Let $C_i(t)$ be the connected component of
$p_{\mu_i(t)}^{-1}({\gamma_i}^t([0,1]))$ containing the critical points of $p_{\mu_i(t)}$
over $a_{\delta}(t)$ and $a_{\varepsilon}(t)$.
We can draw the picture of $C_i(0)$ on $X_0$ in Figure 6.3, then we see
that it is a simple closed curve in $X_0$, and isotopic to $C_i$
if we identify $X_0$ with the genus 3 surfaces in Figure 1.1 by an obvious manner.
The simplicity of the roots of $G^{\mu_i(t)}(x_1)$ for $t\in [0,1)$ implies that the
topological type of $p_{\mu_i(t)}$ is the same as $p_0$, therefore $C_i(t)$ is also a
simple closed curve in $X_{\mu_i(t)}$ for $t\in [0,1)$. On the other hand
$C_i(1)=p_{v_i}^{-1}(a_{\delta}(1))$
consists of a single point, which is a unique singular point of $X_{v_i}$.

\begin{center}
Figure 7.2
\input{figure7-2-1.tex}
\input{figure7-2-2.tex}
\input{figure7-2-3.tex}
\end{center}

Let $D$ be the unit closed disk and choose a continuous family
$\{{\iota_i}^t\colon D \rightarrow \mathbb{C}\}_{0\le t\le 1}$ of embeddings
of $D$ such that ${\iota_i}^t(D)$ contains
${\gamma_i}^t([0,1])$ and does not meet $\{ a_j(t)\}_{j\neq \delta, \varepsilon}$.
Let $A_i(t)$ be the connected component of $p_{\mu_i(t)}^{-1}({\iota_i}^t(D))$
containing $C_i(t)$.
For $t\in [0,1)$, $A_i(t)$ is homeomorphic to an annulus, and $A_i(1)$ is
homeomorphic to the space obtained from an annulus by collapsing
a non null-homologous simple closed curve in it.

Let $M_i$ be the quotient space of $X_0\times [0,1]$ obtained by identifying
all of $C_i(0)\times \{ 1\}$ to a single point.
Using $\{{\iota_i}^t\}_{0\le t\le 1}$, we have a diffeomorphism
$$\bigcup_{0\le t\le 1}\partial A_i(t)\cong \partial A_i(0)\times [0,1]$$
($\partial A_i(t)$ is the boundary of $A_i(t)$)
compatible with the natural projections onto $[0,1]$.
By the observations, we can extended it to a diffeomorphism
\begin{equation}
\label{eq:7-1}
\bigcup_{0\le t\le 1}X_{\mu_i(t)}\setminus {\rm Int}A_i(t)\cong
(X_0\setminus {\rm Int}A_i(0))\times [0,1]
\end{equation}
(${\rm Int}A_i(t)$ is the interior of $A_i(t)$).
Moreover, using $\{{\iota_i}^t\}_{0\le t\le 1}$ again
we can extend (\ref{eq:7-1}) to a homeomorphism from
$\bar{\pi}^{-1}(\mu_i)=\bigcup_{0\le t \le 1}X_{\mu_i(t)}$ to $M_i$
also compatible with the projections onto $[0,1]$.
Here $(X_0\setminus {\rm Int}A_i(0))\times [0,1]$ is understood to be
a subspace of $M_i$ by an obvious manner.

The exsistence of the homeomorphism $\bar{\pi}^{-1}(\mu_i)\cong M_i$
implies that $C_i(0)$ is the vanishing cycle along $\mu_i$.
In summary, we have proved the following.
\begin{prop}
\label{prop:7-1}
The monodromy $\rho^{\prime}(\lambda_i)\in \Gamma_3$ is the right hand Dehn twist along $C_i$. 
\end{prop}

Now we can complete the proof of Theorem \ref{thm:1-1}.

\begin{proof}[Proof of Theorem \ref{thm:1-1}]
We write $L_0$ instead $L=L(7/8,3/4)$ and let $L_1$ be a line of $\mathbb{P}^3$
meeting $X^{\vee}$ transversely. Choose a base point $b_1\in L_1\setminus (L_1\cap X^{\vee})$.
Since the set of all lines of $\mathbb{P}^3$
meeting $X^{\vee}$ transversely is Zariski open hence path connected, there exist
a continuous family $\{ L(t)\}_{t\in [0,1]}$ of lines of $\mathbb{P}^3$
such that $L(t)$ meets $X^{\vee}$ transversely and $L(0)=L_0,L(1)=L_1$.
Let $\mathcal{F}^{\prime}_t:=\{ (x,H)\in \mathbb{P}_3 \times
(L(t)\setminus (L(t)\cap X^{\vee})); x\in H\cap X \}$. Then there exist 
continuous families of homeomorphisms
$\{ \psi_t\colon L_0 \setminus (L_0\cap X^{\vee})\rightarrow L(t)\setminus (L(t)\cap X^{\vee})\}_{0\le t\le 1}$
and $\{ \Psi_t\colon \mathcal{F}^{\prime}_0\rightarrow \mathcal{F}^{\prime}_t\}_{0\le t\le 1}$
such that $\pi_t^{\prime}\circ \Psi_t=\psi_t\circ \pi_0^{\prime}$ where $\pi_t^{\prime}$
is the second projection.
Now $\{ \psi_1(\lambda_i) \}_i$ is a standard generating system for
$\pi_1(L_1 \setminus (L_1\cap X^{\vee}),\psi_1(0))$ such that
the image of $\psi_1(\lambda_i)$ under the associated topological
monodromy is the right hand Dehn twist along $C_i$.
The result follows by considering an isomorphism
$\pi_1(L_1 \setminus (L_1\cap X^{\vee}),\psi_1(0))\cong \pi_1(L_1 \setminus (L_1\cap X^{\vee}),b_1)$
induced by a path from $\psi_1(0)$ to $b_1$.
\end{proof}

\noindent \textbf{Acknowledgments:}
The author is grateful to Professor Nariya Kawazumi for
reading a draft and comments on expositions of the paper.
He is also grateful to Masatoshi Sato for his advice
which made the arguments in Section 7 clear.
This research is supported by JSPS Research
Fellowships for Young Scientists (19$\cdot$5472).

\noindent \textsc{Yusuke Kuno\\
Graduate School of Mathematical Sciences,\\
The University of Tokyo,\\
3-8-1 Komaba Meguro-ku Tokyo 153-0041, JAPAN}

\noindent \texttt{E-mail address:kunotti@ms.u-tokyo.ac.jp}

\end{document}